\documentclass[11pt]{amsart}

\usepackage{amsmath,amssymb,amsthm,mathtools}
\usepackage{microtype}
\usepackage[hidelinks]{hyperref}

\newcommand{\T}{\mathbb T}
\newcommand{\C}{\mathbb C}

\theoremstyle{plain}
\newtheorem{theorem}{Theorem}[section]

\newtheorem{lemma}[theorem]{Lemma}

\theoremstyle{remark}

\title[On the contractivity of the Riesz projection]
{On the contractivity of the Riesz projection}

\author{Xiangdi Fu}
\address{School of Fundamental Physics and Mathematical Sciences, HIAS,
University of Chinese Academy of Sciences, Hangzhou 310024, China}
\email{xdfu@ucas.ac.cn}

\author{Dilong Li}
\address{School of Mathematical Sciences, Fudan University,
Shanghai 200433, China}
\email{dlli23@m.fudan.edu.cn}


\begin{document}

\begin{abstract}
This note presents a proof that the Riesz projection is contractive from $L^q$ to $L^{4(1-1/q)}$ for the case $1<q<\infty$. This fills in the previously unresolved parameter range of \cite[Conjecture 4]{BOSZ18}.

\end{abstract}

\maketitle

\section{The main result}

Let $\T$ be the unit circle in the complex plane, equipped with the normalized Lebesgue measure $dm$. For $0<p\leq \infty$, we denote by $L^p$ the usual Lebesgue space on $\T$. It is well known that the Riesz projection (initially defined on trigonometric polynomials)
$$
\mathcal P\left(\sum_{n\in\mathbb Z}c_nz^n\right)
=\sum_{n\geq 0}c_nz^n
$$
extends to a bounded operator on $L^p$ if and only if $1<p<\infty$. Moreover, the explicit norm 
\begin{align}\label{eq:HV-norm-Riesz-projection}
  \|\mathcal P\|_{L^p\to H^p}= \csc(\pi/p), \quad 1<p<\infty,
\end{align} was determined by Hollenbeck and Verbitsky \cite{HV}. 

In this note, we consider the $L^q$-$L^p$ contractivity of the Riesz projection: \emph{for which values of $q$ and $p$ does $\mathcal P$ extend to a contraction from $L^q$ to $L^p$?} This problem was first investigated by Marzo and Seip \cite{MS11} and has since given rise to several interesting developments and applications. As pointed out in \cite{MS11}, by H\"older's inequality, this contractivity problem essentially amounts to determining the following \emph{critical exponent of the Riesz projection}: \[
\mathfrak p(q):=
\max\left\{
p>0:\mathcal P:L^q\to L^p
\text{ is contractive}
\right\}.
\] 
By orthogonality, one immediately has $\mathfrak p(2)=2$. Moreover, Marzo and Seip \cite{MS11} proved that $\mathfrak p(\infty)=4$ and hence obtained $\mathfrak p(4/3)=1$ from the duality. Later, the upper bound $$\mathfrak p(q)\leq 4\left(1-\frac1q\right), \qquad 1<q\leq \infty$$ was established by Brevig, Ortega-Cerdà, Seip, and Zhao in \cite{BOSZ18}. Motivated by a duality relation between certain contractive inequalities for Hardy spaces, Brevig, Ortega-Cerdà, Seip, and Zhao \cite[Conjecture~4]{BOSZ18} then conjectured that
\begin{align}\label{eq:BOSZ-conjecture}
\mathfrak p(q)=4\left(1-\frac1q\right), \qquad 1<q\leq \infty.
\end{align}
They also observed that Conjecture \eqref{eq:BOSZ-conjecture} predicts that the critical exponent tends to zero as $q$ decreases to one. This naturally led them to conjecture the endpoint inequality
\begin{align}\label{eq:conjecture-endpoint}
\|\mathcal P f\|_{L^0}
:=
\exp\left(\int_{\mathbb T}\log |\mathcal P f|\,dm\right)
\leq
\|f\|_{L^1}.
\end{align}
Recently, Brevig, Llinares, and Seip \cite{BLS} confirmed that the endpoint inequality \eqref{eq:conjecture-endpoint} indeed holds for all $f\in L^1$. They also formulated a conjecture for the critical exponents of Riesz projections in several variables and proved that, if the one-dimensional critical exponent conjecture \eqref{eq:BOSZ-conjecture} is true, then the corresponding conjecture in two dimensions is also true \cite[Corollary 4]{BLS}. However, this implication does not necessarily extend to higher dimensions. We mention that these contractive inequalities have important applications to Hardy spaces on higher-dimensional tori and, via the Bohr lift, to Hardy spaces of Dirichlet series; see, for instance, \cite{Bay02,BBSSZ18,BHS15,BOS21,Hel10,KQSS22,Ku22,Lin24,MS11}.

The main result of this note is to prove that \eqref{eq:BOSZ-conjecture} is valid for all the remaining values $1<q<\infty$.

\begin{theorem}\label{thm:main}
For every $1<q<\infty$, the Riesz projection $$\mathcal P: L^q \to L^{4(1-1/q)}$$ is contractive.
\end{theorem}

Our approach goes back to Cole's work on sharp inequalities for conjugate functions, where best constant inequalities were related to the existence of suitable subharmonic minorants; see \cite{Gam79,HKV03}. This method was later developed in a effective plurisubharmonic form by Hollenbeck and Verbitsky \cite{HV}, and has been widely used to determine exact norms and sharp constants for Riesz projections and Hilbert transforms \cite{Ba78,HKV03,Ose12}. 

\section{The proof of Theorem \ref{thm:main}}

The main difficulty in applying the subharmonic minorant method lies in constructing a suitable subharmonic function adapted to the desired integral inequality. In the following content, we put $$p:= 4\bigg( 1-\frac{1}{q}\bigg).$$

Let us begin by considering the weighted arithmetic-geometric mean inequality:
$$
\lambda a^q+1-\lambda
\geq
a^{\lambda q},\quad 0\leq \lambda \leq 1.
$$
Set $\lambda:= p/q$. Then $0<\lambda\leq 1$ for $1<q<\infty$, and hence
\begin{align}\label{eq:arithmetic-geometric-inequality}
  \lambda a^q+1-\lambda-a^p\geq 0.
\end{align}
Now define $$\Phi(w,z):= \lambda|w+\overline{z}|^q +1 - \lambda -|w|^p, \quad (w,z)\in \mathbb C^2.$$ 

Next, we turn to finding a suitable plurisubharmonic function that bounds $\Phi$ from below. It is clear that $\Phi\geq 0$ whenever $w=0$. Observe that when $w\neq 0$, it holds that
\begin{align*} 
  \Phi(w,z)=&\lambda \frac{\big| |w|^2 + \overline{zw}\big|^q}{|w|^q}+1 -\lambda - |w|^p\\
  \geq & \inf_{a>0}
  \left\{
  \lambda\frac{|a^2+zw|^q}{a^q}
  +1-\lambda-a^p
  \right\}.
\end{align*} 
In particular, if we define 
\begin{align}\label{eq:def-G}G(\zeta):= \inf_{a>0}
  \left\{
  \lambda\frac{|a^2+\zeta|^q}{a^q}
  +1-\lambda-a^p
  \right\},
\end{align}
then
\begin{align}\label{eq:Phi-geq-G}
  \Phi(z,w) \geq G(zw),\quad (z,w)\in \C^2.
\end{align}

  The key point is that, although the infimum of a family of subharmonic functions need not be subharmonic in general, the function $G$ defined here is indeed subharmonic. 

  \begin{lemma}\label{lem:G-subharmonic}
  The function $G$ defined in \eqref{eq:def-G} is subharmonic.
  \end{lemma}

  Let us first prove Theorem \ref{thm:main} on the assumption that Lemma \ref{lem:G-subharmonic} holds, and then proceed to verify Lemma \ref{lem:G-subharmonic} in the next section.

  \begin{proof}[Proof of Theorem \ref{thm:main}]
    Suppose $$f=g+\overline{h}$$ where $g,h$ are analytic polynomials and $h(0)=0$. Since $G$ is subharmonic, the composition $$z\mapsto G \big(g(z) \cdot h(z)\big)$$ is also subharmonic.
    By the sub-mean-value principle, we have 
  \begin{align*}
  &\lambda\int_{\T}|g+\overline h|^q\,dm+1-\lambda
  -\int_{\T}|g|^p\,dm\\
  = & \int_{\T}\Phi(g,h)\,dm
  \geq
  \int_{\T}G ( g\cdot h)\,dm\\
  \geq  & G\big( g(0)h(0)\big)=G(0)=0.
  \end{align*}
  It follows that
  \begin{equation*}
  \int_{\T}|\mathcal P f|^p\,dm
  \leq
  \lambda\int_{\T}|f|^q\,dm+1-\lambda.
  \end{equation*}
  This shows $$\|\mathcal P f\|_{L^p} \leq 1$$ provided that $\|f\|_{L^q} \leq 1$. The proof of Theorem \ref{thm:main} is completed.
  \end{proof}

  \section{Proof of Lemma \ref{lem:G-subharmonic}} 

  It remains to verify that the function $G$ defined in \eqref{eq:def-G} is indeed subharmonic. The case $q=p=2$ is trivial, and in this case $G(\zeta)=2 \operatorname{Re}\zeta$. In the following content we always assume $q\neq 2$, and set
\[
P:=p/2,\qquad Q:=q/2,\qquad \text{and recall that } \lambda=p/q=P/Q.
\]
Then the relation $p=4(1-1/q)$ gives
\begin{equation}\label{eq:critical-PQ}
PQ=2Q-1>0,
\qquad
Q-P=\frac{(Q-1)^2}{Q}>0.
\end{equation}
A change of variable ($x=a^2$) gives $$G(\zeta)=\inf_{x>0} F(x,\zeta),$$ where 
\begin{equation}\label{eq:def-partial-minimum-F}
F(x,\zeta):=
\frac{P}{Q}\frac{|x+\zeta|^{2Q}}{x^Q}
+1-\frac{P}{Q}-x^P,
\qquad (x,\zeta)\in(0,\infty)\times\C.
\end{equation}

Next, we compute the partial derivatives to locate the minimizer in \(x\) of $F(x,\zeta)$ for each fixed \(\zeta\). We adopt the notation $$\partial_1:=\frac{\partial}{\partial x}, \quad \partial_2=\frac{\partial}{\partial \zeta}, \quad \text{and } \overline{\partial}_2: =\frac{\partial}{\partial \overline{\zeta}}.$$

\begin{lemma}\label{lem:minimizer}
For each $\zeta \in \mathbb C$, there exists a unique $$x_{\operatorname{min}}(\zeta) \in (0,\infty)$$ such that $$G(\zeta)=\inf_{x>0} F(x,\zeta) = F\big(x_{\operatorname{min}}(\zeta), \zeta \big).$$ Moreover, the minimizer function $$x_{\operatorname{min}}: \zeta \mapsto x_{\operatorname{min}}(\zeta) $$ is smooth and $x_{\operatorname{min}}(\zeta) > |\zeta|$ for all $\zeta \in \C$.
\end{lemma}

\begin{proof}
Since $2Q>1$, the function $F$ is continuously differentiable
throughout $(0,\infty)\times\C$. A standard calculation gives 
\begin{align}\label{eq:partial-1-F}
\partial_1 F(x,\zeta)
=\begin{cases} 
  P x^{P-1}\bigg(x^{-P-Q}|x+\zeta|^{2Q-2}(x^2-|\zeta|^2) -1 \bigg),  & \text{if $\zeta \neq 0$ and $x\neq -\zeta$,}\\
  -P x^{P-1}, & \text{if $\zeta \neq 0$ and $x=-\zeta$,}\\
  P x^{P-1}\big(x^{Q-P}-1\big), & \text{if $\zeta=0$.}
\end{cases}
\end{align}
In the case $\zeta=0$, since $Q>P$, we see $$x_{\operatorname{min}}(0)=1.$$ For fixed $\zeta\ne0$, we observe that 
$$
\lim_{x\to 0}F(x,\zeta)
=\lim_{x\to\infty}F(x,\zeta)=+\infty.
$$
Thus the minimum is indeed attained, and every minimizer point is a solution of \begin{align}\label{eq:stationary-equation}
  \partial_1 F(x,\zeta)=0.
\end{align}
It follows from \eqref{eq:partial-1-F} that every minimizing point satisfies $x>|\zeta|$. 

Note that for $x>|\zeta|$, the stationary equation \eqref{eq:stationary-equation} is equivalent to
\begin{equation}\label{eq:minimizer-equation}
\frac{|x+\zeta|^{2Q-2}(x^2-|\zeta|^2)}{x^{P+Q}}=1.
\end{equation}
We claim that the left-hand side is strictly increasing in $x$ on
$(|\zeta|,\infty)$. Indeed,
\begin{equation*}
\begin{aligned}
&\partial_1\log
\left(\frac{|x+\zeta|^{2Q-2}(x^2-|\zeta|^2)}{x^{P+Q}}\right)=\operatorname{Re}\frac{2(Q-1)}{x+\zeta}
+\frac{2x}{x^2-|\zeta|^2}
-\frac{P+Q}{x}.
\end{aligned}
\end{equation*}
Note that $$\frac{1}{x+|\zeta|} \leq \operatorname{Re}\frac{1}{x+\zeta} \leq \frac{1}{x-|\zeta|}.$$ Substituting the left-hand estimate when $Q>1$ and the right-hand estimate when $Q<1$, we obtain, after putting the resulting terms over a common
denominator,

\begin{equation*}
\begin{aligned}
&\partial_1\log
\left(\frac{|x+\zeta|^{2Q-2}(x^2-|\zeta|^2)}{x^{P+Q}}\right)\quad\\
&\quad\geq 
\frac{
(Q-P)x^2
-2|Q-1| \, x \, |\zeta|
+(P+Q)|\zeta|^2
}{
x(x^2-|\zeta|^2)
}\\
&\quad=
\frac{
\bigl(|Q-1|x-Q|\zeta|\bigr)^2
+(2Q-1)|\zeta|^2
}{
Qx(x^2-|\zeta|^2)
}
>0.
\end{aligned}
\end{equation*}
The last equality uses the identities in \eqref{eq:critical-PQ}, and the strict
inequality follows from $x>|\zeta|$ and $2Q>1$. Thus
\eqref{eq:stationary-equation} has the unique solution $$x=x_{\operatorname{min}}(\zeta)>|\zeta|.$$ Moreover, it is clear that $F$ is smooth on the domain $$\mathcal R:= \Big\{(x,\zeta)\in (0,\infty)\times \C: x>|\zeta|\Big\}.$$ Since the left-hand side of \eqref{eq:minimizer-equation} is strictly increasing with $x$, it follows that $$\partial^2_1 F(x_{\operatorname{min}}(\zeta),\zeta)>0,\quad \zeta \in \C.$$ Since $\big(x_{\operatorname{min}}(\zeta),\zeta\big)\in \mathcal R$ for every $\zeta \in\C$, an application of the implicit function theorem to the equation \eqref{eq:stationary-equation} shows that $x_{\operatorname{min}}$ is smooth.
\end{proof}

By Lemma \ref{lem:minimizer}, we conclude that $$G(\zeta)=F(x_{\operatorname{min}}(\zeta), \zeta ),$$ and hence $G$ is also smooth. Thus it remains to show $$\Delta G \geq 0.$$ We now calculate the Laplacian directly. To simplify the notation, we write
$$\partial:=\frac{\partial}{\partial \zeta},\qquad \overline{\partial}:=\frac{\partial}{\partial \overline{\zeta}}.$$ Moreover, in the calculation below, $x$ always
denotes $x_{\operatorname{min}}(\zeta)$, and every partial derivative of $F$
is evaluated at the point $(x_{\operatorname{min}}(\zeta),\zeta)$.

\begin{proof}[Proof of Lemma \ref{lem:G-subharmonic}]

Differentiating 
\begin{equation}\label{eq:stationary-equation-at-minimizer}
\partial_1F(x,\zeta)=0
\end{equation}
with respect to $\zeta$ and $\overline\zeta$,
respectively, and using $\partial_1^2F(x,\zeta)>0$, we obtain
\begin{equation}\label{eq:derivatives-of-minimizer}
\partial x_{\operatorname{min}}(\zeta)
=-\frac{\partial_1\partial_2F(x,\zeta)}
{\partial_1^2F(x,\zeta)},
\qquad
\bar\partial x_{\operatorname{min}}(\zeta)
=-\frac{\partial_1\overline{\partial}_2F(x,\zeta)}
{\partial_1^2F(x,\zeta)}.
\end{equation}
On the other hand, the chain rule and
\eqref{eq:stationary-equation-at-minimizer} give
\[
\partial G(\zeta)
=\partial_1F(x,\zeta)\,\partial x_{\operatorname{min}}(\zeta)
+\partial_2F(x,\zeta)
=\partial_2F(x,\zeta).
\]
Applying $\bar\partial$ to this identity and then using
\eqref{eq:derivatives-of-minimizer}, we find that
\begin{equation}\label{eq:direct-second-derivative-of-G}
\begin{aligned}
\partial\bar\partial G(\zeta)
&=\partial_1\partial_2F(x,\zeta)\,
  \bar\partial x_{\operatorname{min}}(\zeta)
  +\partial_2\overline{\partial}_2F(x,\zeta)\\
&=\partial_2\overline{\partial}_2F(x,\zeta)
-\frac{
\partial_1\partial_2F(x,\zeta)\,
\partial_1\overline{\partial}_2F(x,\zeta)
}{\partial_1^2F(x,\zeta)}\\
&=\frac{
\partial_1^2F(x,\zeta)\,
\partial_2\overline{\partial}_2F(x,\zeta)
-|\partial_1\partial_2F(x,\zeta)|^2
}{\partial_1^2F(x,\zeta)}.
\end{aligned}
\end{equation}
Here we used the fact that $F$ is real-valued, and hence
$\partial_1\overline{\partial}_2F=
\overline{\partial_1\partial_2F}$.

Now it suffices to show the numerator in the last expression is nonnegative. Applying $\partial_1$ to \eqref{eq:partial-1-F} and then substituting \eqref{eq:minimizer-equation}, we obtain
\begin{align*}
\partial_1^2F(x,\zeta)=P x^{P-1}
\left(
\frac{2x}{x^2-|\zeta|^2}
+\operatorname{Re}\frac{2(Q-1)}{x+\zeta}
-\frac{P+Q}{x}
\right).
\end{align*}
Similarly, applying $\partial_2$ to \eqref{eq:partial-1-F} and then substituting \eqref{eq:minimizer-equation} gives
\begin{align*}
\partial_1\partial_2F(x,\zeta)=P x^{P-1}
\left(
\frac{Q-1}{x+\zeta}
-\frac{\overline\zeta}{x^2-|\zeta|^2}
\right).
\end{align*}
Finally, direct differentiation of \eqref{eq:def-partial-minimum-F} gives
\begin{align*}
\partial_2F(x,\zeta)
&=P x^{-Q}|x+\zeta|^{2Q-2}(x+\overline\zeta),\\
\partial_2\overline{\partial}_2F(x,\zeta)
&=PQx^{-Q}|x+\zeta|^{2Q-2}
=\frac{PQx^P}{x^2-|\zeta|^2},
\end{align*}
where the last equality again follows from
\eqref{eq:minimizer-equation}.

Substituting these three formulas into \eqref{eq:direct-second-derivative-of-G} gives
\begin{align*}
&\frac{
\partial_1^2F(x,\zeta)\,
\partial_2\overline{\partial}_2F(x,\zeta)
-|\partial_1\partial_2F(x,\zeta)|^2
}{P^2x^{2P-2}}\\
&\quad=
\frac{Qx}{x^2-|\zeta|^2}
\left(
\frac{2x}{x^2-|\zeta|^2}
+2(Q-1)\operatorname{Re}\frac{1}{x+\zeta}
-\frac{P+Q}{x}
\right)\\
&\qquad
-\left|
\frac{Q-1}{x+\zeta}
-\frac{\overline\zeta}{x^2-|\zeta|^2}
\right|^2\\
&\quad=
\frac{2Qx^2-|\zeta|^2}{(x^2-|\zeta|^2)^2}
-\frac{Q(P+Q)}{x^2-|\zeta|^2}\\
&\qquad + \underbrace{ \frac{2Q(Q-1)x}{x^2-|\zeta|^2}
  \operatorname{Re}\frac{1}{x+\zeta}
-\frac{(Q-1)^2}{|x+\zeta|^2}
+\frac{2(Q-1)}{x^2-|\zeta|^2}
  \operatorname{Re}\frac{\zeta}{x+\zeta}}_{\text{denoted by }\mathcal T}.
\end{align*}
Since
\[
\operatorname{Re}\frac{1}{x+\zeta}
=\frac{x+\operatorname{Re}\zeta}{|x+\zeta|^2},
\qquad
\operatorname{Re}\frac{\zeta}{x+\zeta}
=\frac{x\operatorname{Re}\zeta+|\zeta|^2}{|x+\zeta|^2},
\]
the last three terms in the preceding display can be combined as 
\begin{align*}
\mathcal T&=\frac{Q-1}{(x^2-|\zeta|^2)|x+\zeta|^2}
\Big(
2Qx(x+\operatorname{Re}\zeta)
-(Q-1)(x^2-|\zeta|^2)\\
&\hspace{5mm}
+2(x\operatorname{Re}\zeta+|\zeta|^2)
\Big)\\
&=
\frac{Q-1}{(x^2-|\zeta|^2)|x+\zeta|^2}
\Big(
(Q+1)x^2+2(Q+1)x\operatorname{Re}\zeta
+(Q+1)|\zeta|^2
\Big)\\
&=\frac{Q^2-1}{x^2-|\zeta|^2}.
\end{align*}
As a consequence,
\begin{align*}
&\frac{
\partial_1^2F(x,\zeta)\,
\partial_2\overline{\partial}_2F(x,\zeta)
-|\partial_1\partial_2F(x,\zeta)|^2
}{P^2x^{2P-2}}\\
&\quad=
\frac{2Qx^2-|\zeta|^2}{(x^2-|\zeta|^2)^2}
-\frac{Q(P+Q)}{x^2-|\zeta|^2}
+\frac{Q^2-1}{x^2-|\zeta|^2}\\
&\quad=
\frac{
(2Q-1-PQ)x^2+PQ|\zeta|^2
}{(x^2-|\zeta|^2)^2}\\
&\quad=
\frac{PQ|\zeta|^2}{(x^2-|\zeta|^2)^2}
\geq 0,
\end{align*}
where the last equality follows from $2Q-1-PQ=0$ in \eqref{eq:critical-PQ}. Since
$\partial_1^2F(x,\zeta)>0$, formula
\eqref{eq:direct-second-derivative-of-G} now shows that
\[
\partial\bar\partial G(\zeta)
=\frac{
P^2x^{2P-2}PQ|\zeta|^2
}{
(x^2-|\zeta|^2)^2\partial_1^2F(x,\zeta)
}
\geq0.
\]
Thus $G$ is subharmonic. The proof is completed.
\end{proof}

\smallskip\noindent\textbf{AI disclosure.} This work was developed with the assistance of OpenAI's ChatGPT-5.6 Sol. The authors originally proposed seeking a proof of \cite[Conjecture 4]{BOSZ18} by applying Cole's subharmonic minorant method and supplied ChatGPT with several papers in which related subharmonic methods are successfully used to determine exact operator norms. After the authors interacted with ChatGPT over several rounds, ChatGPT identified the suitable subharmonic minorant and supplied a complete proof. The authors subsequently verified the proof independently and revised its presentation primarily to improve readability, making only minor simplifications to the mathematical argument. The authors take full responsibility for the contents of this manuscript.

\smallskip\noindent\textbf{Acknowledgements.} Dilong Li was supported by the  National Natural Science Foundation of China (Grant No. 124B2005).

\end{document}